\documentclass[10 pt,reqno]{amsart}
 \usepackage{amsmath,amssymb}
\usepackage{latexsym}
\numberwithin{equation}{section}
\newtheorem{theorem}{Theorem}[section]

\newtheorem{remark}[theorem]{Remark}

\newtheorem{lemma}[theorem]{Lemma}
\newtheorem{proposition}[theorem]{Proposition}
\newcommand{\dk}{\mathbf{k}}
\title[Long-time dynamics of weakly nonlinear CGL]{ Long-time Dynamics of   Resonant Weakly Nonlinear   CGL Equations }
\author{HUANG Guan\\C.M.L.S, \'Ecole Polytechnique}
\begin{document}
\maketitle
\begin{abstract}
Consider a weakly nonlinear CGL equation on the torus~$\mathbb{T}^d$:
\[u_t+i\Delta u=\epsilon [\mu(-1)^{m-1}\Delta^{m} u+b|u|^{2p}u+ ic|u|^{2q}u].\eqno{(*)}\]
Here $u=u(t,x)$, $x\in\mathbb{T}^d$, $0<\epsilon<<1$, $\mu\geqslant0$, $b,c\in\mathbb{R}$ and $m,p,q\in\mathbb{N}$. Define \mbox{$I(u)=(I_{\dk},\dk\in\mathbb{Z}^d)$},  where $I_{\dk}=v_{\dk}\bar{v}_{\dk}/2$ and  $v_{\dk}$, $\dk\in\mathbb{Z}^d$, are the Fourier coefficients of the function~$u$ we give. Assume that  the equation $(*)$  is well posed on time intervals of order $\epsilon^{-1}$ and its solutions have there  a-priori bounds, independent of the small parameter. Let $u(t,x)$ solve the equation $(*)$.  If $\epsilon$ is small enough, then for $t\lesssim\epsilon^{-1}$,  the quantity $I(u(t,x))$ can be well described by solutions of an {\it effective equation}:
\[u_t=\epsilon[\mu(-1)^{m-1}\Delta^m u+ F(u)],\]
where the term $F(u)$ can be constructed through a kind of resonant  averaging of the nonlinearity $b|u|^{2p}+ ic|u|^{2q}u$.
\end{abstract}
\setcounter{section}{-1}
\bibliographystyle{plain}
\section{Introduction}
We consider a weakly nonlinear complex Ginzburg-Landau (CGL) equation on the \mbox{$d$-torus} \mbox{$\mathbb{T}^d=\mathbb{R}^d/2\pi l\mathbb{Z}^d$}:
\begin{equation}
u_t+i\Delta u= \epsilon[\mu(-1)^{m-1}\Delta^m u+b|u|^{2p}u+ ic |u|^{2q}u], \quad u=u(t,x),\; x\in\mathbb{T}^d,
\label{rnls1}
\end{equation}
where $m,p,q\in\mathbb{N}$, $\mu\geqslant0$, $b,c\in\mathbb{R}$ and $\epsilon$ is a small parameter.   The {\it Large box limit} of the 2d cubic  NLS equation, which is   Eq. (\ref{rnls1}) with $\mu=b=0$, $q=2$ and $d=2$,   was studied in \cite{fgh2013}.  There, by simultaneously letting  $l\to\infty$ and $\epsilon\to0$  with the relation $l\ll\epsilon^{-1}$ satisfied,  an equation that describes the large box limit  was derived. While the existence of this double limit is quite restricted to the case of 2d cubic NLS, the  weakly nonlinear limit, derived by letting $\epsilon\to0$ with constant~$l$, may exist for many nonlinear PDEs.  In this work\footnote{The result of the present paper is part of the Ph.D works of the author at \'Ecole Polytechnique, France} we will calculate  the weakly nonlinear limit for  the CGL equation (\ref{rnls1}). We  mention that our approach as well applies to study weakly nonlinear wave equations.

For simplicity we  fix the period of the torus $\mathbb{T}^d$ to be 1. For any $s\in\mathbb{R}$ denote by $H^s$ the Sobolev space of complex-valued periodic functions, provided with the norm $||\cdot||_s$,
\[||u||_s^2=\big\langle (-\Delta )^s u,u\big\rangle+\langle u, u\rangle\quad \text{if}\quad s\in\mathbb{N}
\]
where $\langle \cdot,\cdot\rangle$ is the real scalar product in $L^2(\mathbb{T}^d)$,
\[\langle u,v\rangle=Re\int_{\mathbb{T}^d}u\bar{v}dx,\quad u,v\in L^2(\mathbb{T}^d).
\]
When $\mu=b=0$,  Eq. (\ref{rnls1}) is hamiltonian and may be written as 
\[u_t=i\partial_u \mathcal{H}(u),\quad \mathcal{H}(u)=\int_{\mathbb{T}^d}\big[\frac{1}{2}|\nabla u|^2+ \frac{c\epsilon }{2q+2}|u|^{2q+2}\big]dx,\]
where $\partial_u$ stands for the $L^2$-gradient with respect to $u$.

 We  mainly study the long-time dynamics of Eq.  (\ref{rnls1}) on time intervals of order $\epsilon^{-1}$ when  $0<\epsilon\ll1$. We assume:
\smallskip

\noindent {\bf Assumption A}: {\it There exist some $s>d/2$ and $T>0$ such that for every $u_0\in H^s$, Eq. (\ref{rnls1}) has a unique solution $u(t,x)\in C([0,\epsilon^{-1}T], H^s)$ with initial datum $u_0$, and $||u(t,x)||_s\leqslant C(||u_0||_s, T)$ for $t\leqslant \epsilon^{-1}T$. }
\medskip

Many sufficient conditions for this assumption are known. For example:
\begin{proposition}1)  (See \cite{brezis1980,bourgain1993,herr2011}) If $\mu=b=0,$ then
 Assumption A holds  for
\begin{equation}q\in \mathbb{N}, \; q<+\infty, \; \text{when}\;\; d=1,2 \quad \text{and}\quad q=1, 2,\;\text{when}\;\;d=3.\label{rnls-p1-r0}\end{equation}
2) If $\mu>0$, $b\leqslant0$ and $c\leqslant0$, then for any fixed $p,q,d\in\mathbb{N}$, there exists $m\in\mathbb{N}$ such that Assumption A holds.
\label{rnls-p1}\end{proposition}
It is  convenient to introduce   the slow time $\tau=\epsilon t$. 
Writing Eq. (\ref{rnls1}) in~$\tau$, we get the rescaled equation
\begin{equation}
\dot{u}+\epsilon^{-1}i\Delta u=\mu(-1)^{m-1}\Delta^m u+b|u|^{2p}u+ ic |u|^{2q}u,
\label{rnls2}
\end{equation}
where $u=u(\tau,x)$, $x\in\mathbb{T}^d$ and the dot $\dot{}$ stands for $\frac{d}{d\tau}$.

For  a complex function $u(x)$ on $\mathbb{T}^d$ we define
\begin{equation}\mathcal{F}(u)=(v_{\dk},\;\dk\in\mathbb{Z}^d),
\label{Fourier}
\end{equation}
 where the vector $(v_{\dk},\dk\in\mathbb{Z}^d)$ is formed by the Fourier coefficients of $u$:  $$u(x)=\sum_{\dk\in\mathbb{Z}^d}v_{\bf{k}}e^{i\dk\cdot x},\quad v_{\dk}=\int_{\mathbb{T}^d}u(x)e^{-i\dk\cdot x}dx.$$ 
In the space of complex sequence $v=(v_{\dk},\dk\in\mathbb{Z}^d)$, we introduce the norm:
$$ |v|^2_s=\sum_{\dk\in\mathbb{Z}^d}(|\dk|^{2s}+ 1)|v_{\dk}|^2,\quad s\in\mathbb{R},$$
and denote $h^s=\{v:\; |v|_s<\infty\}$. Obviously, for $s\in\mathbb{R}$, $h^s=\mathcal{F}(H^s)$.

Eq. (\ref{rnls2}) has a rather transparent form in the space $h^s$. Let $u(\tau,x)$ be its solutions, then the Fourier coefficients $v_{\dk}(\tau)$ of $u(\tau,x)$ solves the infinite dimensional  system:
\begin{equation}
\dot{v}_{\dk}-\epsilon^{-1}i\lambda_{\dk}v_{\dk}=-\mu\lambda_{\dk}^mv_{\dk}+bP_{\dk}(v,p)+ icP_{\dk}(v,q), 
\label{rnls-v}
\end{equation}
where $\dk\in\mathbb{Z}^d$, $\lambda_{\dk}=|\dk|^2$ and for every $n\in\mathbb{N}$,
$$P_{\dk}(v,n)=\sum_{(\dk_1,\dots,\dk_{2n+1})\in \mathcal{S}(\dk,n)}v_{\dk_1}\bar{v}_{\dk_2}\cdots v_{\dk_{2n-1}}\bar{v}_{\dk_{2n}}v_{\dk_{2n+1}}$$ with $$\mathcal{S}(\dk,n)=\{(\dk_1,\dots,\dk_{2n+1})\in(\mathbb{Z}^d)^{2n+1}:\;\sum_{j=1}^{2n+1}(-1)^{j-1}\dk_j=\dk\}.$$
We denote $\Lambda=(\lambda_{\dk},\dk\in\mathbb{Z}^d)$ and  call it the {\it frequency vector} of Eq. (\ref{rnls2}).

For every $\dk\in\mathbb{Z}^d$, denote $I_{\dk}=\frac{1}{2}v_{\dk}\bar{v}_{\dk}$ and $\varphi_{\dk}=\text{Arg}\; v_{\dk}$. Notice that the quantities~$I_{\dk}$ are conservation laws of the linear equation $(\ref{rnls1})_{\epsilon=0}$. We call them the action variables and correspondingly, call the quantities $\varphi_{\dk}$ the angle variables. 
We introduce the weighted $l^1$-space $h_I^s$:
$$h_I^s:=\{I=(I_{\dk},\dk\in\mathbb{Z}^d)\in \mathbb{R}^{\infty}: |I|_s^{\sim}=\sum_{\dk\in\mathbb{Z}^d}2(|\dk|^{2s}+ 1)|I_{\dk}|<\infty\}.$$

 Using the action-angle variables $(I,\varphi)$, we can write equation (\ref{rnls-v}) as a slow-fast system:
\[
\dot{I}_{\dk}=v_{\dk}\cdot [-\mu\lambda_{\dk}^mv_{\dk}+P_{\dk}(v,p)+iP_{\dk}(v,q)],\quad \dot{\varphi}_{\dk}=\epsilon^{-1}\lambda_{\dk}+|v_{\dk}|^{-2}\cdots,\quad \dk\in\mathbb{Z}^d.
\]
Here  the dots stand for a term of order 1 (as $\epsilon\to0$).
Our task is  to study the evolution of actions $I_{\dk}$ when $\epsilon\ll 1$.  Due to the polynomial form of the nonlinearity,  there exists an effective way to deal with this problem. That is  the so-called {\it interaction representation picture}.   Let us define
\[a_{\dk}(\tau)=e^{-i\epsilon^{-1}\lambda_{\dk}\tau}v_{\dk}(\tau).\]
Clearly, $|a_{\dk}|^2=|v_{\dk}|^2=I_{\dk}/2$. Therefore the limiting behaviour (as $\epsilon\to0$) of the quantity $|a_{\dk}|$
characterizes the limiting behaviour of the action variables  $I_{\dk}$. Using Eq. (\ref{rnls-v}), we obtain the equation satisfied by $a_{\dk}(\tau)$:
\[\begin{split}
\dot{a}_{\dk}(\tau)=&-\mu\lambda_{\dk}^ma_{\dk} \\
&+b\sum_{(\dk_{1},\dots,\dk_{2p+1})\in S(\dk,p)}a_{\dk_1}(\tau)\overline{a_{\dk_2}(\tau)}\cdots a_{\dk_{2p-1}}(\tau)\overline{a_{\dk_{2p}}(\tau)}a_{\dk_{2p+1}}(\tau)\\
&\qquad\qquad\qquad\times \exp\{i\epsilon^{-1}\tau[-\lambda_{\dk}+\sum_{j=1}^{2p+1}(-1)^{j-1}\lambda_{\dk_{j}}]\}\\
&+ci\sum_{(\dk_{1},\dots,\dk_{2q+1})\in S(\dk,q)}a_{\dk_1}(\tau)\overline{a_{\dk_2}(\tau)}\cdots a_{\dk_{2q-1}}(\tau)\overline{a_{\dk_{2q}}(\tau)}a_{\dk_{2q+1}}(\tau)\\
&\qquad\qquad\qquad\times \exp\{i\epsilon^{-1}\tau[-\lambda_{\dk}+\sum_{j=1}^{2q+1}(-1)^{j-1}\lambda_{\dk_{j}}]\},
\end{split}
\]
where $\dk\in\mathbb{Z}^d$. The terms in the right hand side oscillate  fast if $\epsilon $ is  small, except the terms for which  the sum in the exponential equals zero. This leads to the guess that only these terms determine the limiting  behavior of $a_{\dk}(\tau)$ as $\epsilon\to0$, and that the resulting dynamics is controlled by the following system: for $\dk\in\mathbb{Z}^d,$
\begin{equation}
\dot{a}_{\dk}(\tau)=-\mu\lambda_{\dk}^ma_{\dk}+b\mathbf{R}_{\dk}(a,p)+ci \mathbf{R}_{\dk}(a,q),
\label{rnls-effective2}
\end{equation}
where    for every $n\in\mathbb{N}$,
$$\mathbf{R}_{\dk}(a,n)=\sum_{(\dk_{1},\dots,\dk_{2n+1})\in \mathcal{R}(\dk,n)}a_{\dk_1}(\tau)\overline{a_{\dk_2}(\tau)}\cdots a_{\dk_{2n-1}}(\tau)\overline{a_{\dk_{2n}}(\tau)}a_{\dk_{2n+1}}(\tau),$$
with
$$\mathcal{R}(\dk,n):=\{(\dk_1,\dots,\dk_{2n+1})\in S(\dk,n):\; -\lambda_{\dk}+\sum_{j=1}^{2n+1}(-1)^{j-1}\lambda_{\dk_{j}}=0\;\}.$$  We call Eq.  (\ref{rnls-effective2}) the {\it effective equation} for Eq.  (\ref{rnls2}).
We will see  in Section 1 that  it can be defined by an averaging process and is well posed in the spaces $h^s$, $s>d/2$. 

When $\mu=b=0$, Eq. (\ref{rnls-effective2})  is hamiltonian  with the Hamiltonian function:
\[\mathcal{H}_{res}(v)=\frac{c}{2q+2}\sum_{(\dk_1,\dots,\dk_{2q+2})\in \mathcal{RES}}v_{\dk_1}\bar{v}_{\dk_2}\cdots v_{\dk_{2q+1}}\bar{v}_{\dk_{2q+2}},
\]
where  $$\mathcal{RES}:=\{(\dk_1,\dots,\dk_{2q+2})\in (\mathbb{Z}^d)^{2q+2}: \sum_{j=1}^{2q+2}(-1)^{j-1}\lambda_{\dk_j}=0\}.$$
 Besides the Hamiltonian $\mathcal{H}_{res}$,  this equation possess two  extra  integrals of motion:
$$H_1(v)=\sum_{\dk\in\mathbb{Z}^d}|v_{\dk}|^2,\quad H_2(v)=\sum_{\dk\in\mathbb{Z}^d}\lambda_{\dk}|v_{\dk}|^2.$$

The main result of this work is the following theorem, where $u(t,x)$ is  a solution of Eq. (\ref{rnls1}), $v(\tau)=\mathcal{F}(u(\epsilon^{-1}\tau,x))$ and $a^{\prime}(\tau)$ is a solution of the effective equation~(\ref{rnls-effective2}) with the same initial datum $v(0)$.
\begin{theorem}\label{rnls-main}If Assumption A holds, then the  solution $a'(\tau)$  exists for $0\leqslant\tau\leqslant T$, and   for  every $s_1\in(d/2,s)$ and sufficiently small  $\epsilon$ we have
\begin{equation}|I(v(\tau))-I(a^{\prime}(\tau))|_{s_1}^{\sim}\leqslant C[\epsilon^{\frac{s-s_1}{2m}}+\epsilon^{1/2}],\quad \tau\in [0,T],\label{rnls-main-e}\end{equation}
where the constant $C$ depend only on $s_1$,$s$, $T$ and  the size of the initial datum~$|v(0)|_s$.
\end{theorem}
\begin{remark}1)If $\mu=0$ we can choose $s_1=s$ and the error  in the r.h.s of (\ref{rnls-main-e}) is of order $\epsilon^{1/2}$. Moreover, in this case, if in addition,  the $H^s$-norm of the solution $u(x,t)$ grows as $||u(x,t)||_s\lesssim e^{\epsilon t C(||u(0)||_s)}$, Theorem \ref{rnls-main} can be extended to time intervals of order $\epsilon^{-1}\log \epsilon^{-1}$ with the error $\epsilon^{\alpha}$ for  certain $\alpha>0$.

2) The method of this paper also applies to 2d nonlinear Schr\"odinger equations \mbox{($a=b=0$ and $d=2$ in Eq. (\ref{rnls1}))} with other polynomial nonlinearities, e.g.  the nonlinearities with Hamiltonians \mbox{$\mathcal{H}_3=\int |u|^2(u+\bar{u})dx$} and $\mathcal{H}_3^{\prime}=\int (u^3+\bar{u}^3)dx$.
\end{remark}

Equations that are similar to the effective equation (\ref{rnls-effective2}) recently appear in a number of works, e.g. a stochastic damp-driven version of it is constructed in \cite{skam2013}, using the same philosophy as the present paper. Some of our lemmas are borrowed from that work.  In \cite{pgsg2012}, similar equation is used to derive an effective  equation for a 1d wave equation (which turns out to be an integrable system). The  equations, similar to Eq. (\ref{rnls-effective2}) also are known in the theories of {\it wave turbulence}. There, it is called the {\it equation of discrete turbulence}, see \cite{naz2011}, Chapter 12.   We believe that our result   provides a useful insight in the related topics.

The paper is organized  as follows: In Section 1, we  discuss the effective equation~(\ref{rnls-effective2}). Theorem \ref{rnls-main} is proved in Section 2.  Finally, in Section 3, we discuss the validity of Proposition \ref{rnls-p1}.

 \section{The Effective system}

 Consider the Fourier transform $\mathcal{F}$ (\ref{Fourier}) for complex functions on $\mathbb{T}^d$. Then $|\mathcal{F}u|_s=||u||_s$, for every $s\in\mathbb{R}$.  We denote $\mathcal{F}u=v$ and write  Eq. (\ref{rnls2}) in  the $v$-variables:
\begin{equation}
\dot{v}_{\dk}-\epsilon^{-1}i\lambda_{\dk} v_{\dk}=-\mu\lambda_{\dk}^mv_{\dk}+P_{\dk}(v), \quad \dk\in\mathbb{Z}^d.
\end{equation}
Here $P_{\dk}$ is the  coordinate  component of the mapping $P(v)$ defined by 
$$P(v)=\mathcal{F}(b|u|^pu+c i|u|^{2q}u), \quad u=\mathcal{F}^{-1}(v).$$
 This mapping is analytic of polynomial growth:

\begin{lemma}
The mapping $P(v)$ is an analytic transform of the space $h^s$ with \mbox{$s>d/2$}. Moreover  the norm of $P(v)$ and that of  its  differential   $dP(v)$ have polynomial growth with respect to $|v|_s$.
\label{rnls-lem-lip0}
\end{lemma}
The assertion follows from the well known fact that the spaces $h^s$, $s>d/2$, are Hilbert algebras.

Let $$R(v)=(R_{\dk}(v),\;\dk\in\mathbb{Z}^d)$$ with $R_{\dk}(v)=b\mathbf{R}_{\dk}(v,p)+ci\mathbf{R}_{\dk}(v,q)$, where $\mathbf{R}_{\dk}(v,p)$ and $\mathbf{R}_{\dk}(v,q)$ are the  quantities in the right side of Eq. (\ref{rnls-effective2}) with the notation $a$ replaced by $v$.  For each \mbox{$\theta=(\theta_{\dk},\dk\in\mathbb{Z}^d)\in \mathbb{T}^{\infty}$}, denote by $\Phi_{\theta}$ the linear operator in $h^s$:
\[\Phi_{\theta}(v)=v',\quad v'_{\dk}=e^{i\theta_{\dk}}v_{\dk}.\]
Then 
\begin{lemma} For $v\in h^s$, $s>d/2$, we have 
\begin{equation}
R(v)=\frac{1}{2\pi}\int_0^{2\pi}\Phi_{-t\Lambda}P(\Phi_{t\Lambda}v)dt.
\label{rnls-resonant-av}
\end{equation}
\label{rnls-lemma-ra}
\end{lemma}
\begin{proof}

 For the quantity $P_{\dk}(v,q)$ we have
 \[\begin{split}&\frac{1}{2\pi}\int_0^{2\pi}e^{-i\lambda_{\dk}t}P_{\dk}(\Phi_{t\Lambda}(v),q)dt\\
 &=\frac{1}{2\pi}\int_0^{2\pi}\sum_{(\dk_1,\dots,\dk_{2q+1})\in\mathcal{S}(\dk,q)}\Big\{v_{\dk_1}\bar{v}_{\dk_2}\cdots v_{\dk_{2q-1}}\bar{v}_{\dk_{2q}}v_{\dk_{2q+1}}\\
 &\qquad\qquad\times\exp[-it(-\lambda_{\dk}+\sum_{j=1}^{2q+1}(-1)^{j-1}\lambda_{\dk_j})\big]\Big\}dt\\
 &=\sum_{(\dk_1,\dots,\dk_{2q+1})\in\mathcal{R}(\dk)}v_{\dk_1}\bar{v}_{\dk_2}\cdots v_{\dk_{2q-1}}\bar{v}_{\dk_{2q}}v_{\dk_{2q+1}}=\mathbf{R}_{\dk}(v,q).
 \end{split}\]
The same equality also holds for $P_{\dk}(v,p)$. 
Since $$P_{\dk}(v)=bP_{\dk}(v,p)+ciP_{\dk}(v,q),$$
then  the equality (\ref{rnls-resonant-av}) follows.
\end{proof}
\begin{lemma} The vector field $R(v)$ is locally Lipschitz in Hilbert spaces $h^s$, $s>d/2$.\label{rnls-lem-lip}\end{lemma}
\begin{proof}Let $v_1,v_2\in h^s$ and $|v_1|_s$, $|v_2|_s\leqslant M$.  Then  using Lemmas \ref{rnls-lem-lip0}, \ref{rnls-lemma-ra} and the fact that the operators~$\Phi_{t\Lambda}$, $t\in\mathbb{R}$ define isometries in $h^s$, we have 
\[\begin{split} |R(v_1)-R(v_2)|_s&\leqslant \frac{1}{2\pi}\int_0^{2\pi}|\Phi_{-t\Lambda}[P(\Phi_{t\Lambda}v_1)-P(\Phi_{t\Lambda}v_2)]|_sdt\\
&\leqslant  \frac{1}{2\pi}\int_0^{2\pi}C(M)|\Phi_{t\Lambda}(v_1-v_2)|_sdt\leqslant C(M)|v_1-v_2|_s.\end{split}\]This proves the assertion of the lemma.\end{proof}
From Lemma \ref{rnls-lem-lip}, we know that the effective  equation (\ref{rnls-effective2}) is well posed, at least locally,  in the space $h^s$, $s>d/2$.  When $a=b=0$, it is easy to see from Lemma~\ref{rnls-lemma-ra} that the quantities $||v||_0^2$ and $||v||_1^2$ are integral of motions for Eq. (\ref{rnls-effective2}).

\section{The Main theorem}
In this section we will prove   Theorem \ref{rnls-main}.  Fix $s>d/2$. Denote 
\[ B(M)=\{v\in h^s: |v|_s\leqslant M\},\quad \forall M>0.\]
Fix a $M_0>0$. Let  $u(\tau,x)$ be a solution of Eq. (\ref{rnls2}) such that $$||u(0,x)||_s\leqslant M_0,$$  and $$v(\tau)=\mathcal{F}(u(\tau,x)).$$ Without loss of generality, suppose that Assumption A holds with $T=1$. Then there exists $M_1\geqslant M_0$ such that 
\[v(\tau)\in B( M_1),\quad\tau\in[0,1].\]
Let $$a(\tau)=\Phi_{-\tau\epsilon^{-1}\Lambda }(v(\tau)).$$ Then $a(\tau)$ is the interaction representation picture of~$v(\tau)$. For every \mbox{$v=(v_{\dk},\dk\in\mathbb{Z}^d)$}, denote $$\mathbb{F}(v)=(-\mu\lambda_{\dk}^mv_{\dk},\dk\in\mathbb{Z}^d).$$ We have 
\begin{equation}
\dot{a}(\tau)=\mathbb{F}(a(\tau))+\Phi_{-\tau \epsilon^{-1}\Lambda}\Big(P\big(\Phi_{\tau\epsilon^{-1}\Lambda}(a(\tau))\big)\Big):=\mathbb{F}(a(\tau))+Y\big(a(\tau),\tau\big).
\end{equation}
Let $s_1\in(d/2,s]$. 
Using Lemma \ref{rnls-lem-lip0} and the fact the the operators $\Phi_{t\Lambda}$, $t\in\mathbb{R}$ define isometries on $h^{s_1}$, we get for any $v,\; v^{\prime}\in B(M_1)$ and $\tau\in\mathbb{R}$,
\begin{equation}
|Y(v, \tau)|_{s_1}\leqslant C(s_1,M_1),\quad |Y(v,\tau)-Y(v^{\prime},\tau)|_{s_1}\leqslant C(s_1,M_1)|v-v^{\prime}|_{s_1}.
\label{rnls-r-lip2}
\end{equation}
Denote by $e^{\mathbb{F}t}$  the continuous one parameter group generated by the operator $\mathbb{F}$.
The following lemma is well known in the theories of parabolic PDEs, see e.g. Section 2.1 in \cite{Chu}.
\begin{lemma}
Let $s_1\leqslant s$, then for any $ 0\leqslant\tau_1\leqslant \tau_2$, we have 
\[\begin{split}&|e^{\mathbb{F}\tau_2}v-e^{\mathbb{F}\tau_1}v|_{s_1}\leqslant C(s_1,s,\mu)|\tau_2-\tau_1|^{\frac{s-s_1}{m}}|v|_s,\\
&|\mathbb{F}e^{\mathbb{F}\tau}v|_{s_1}\leqslant \big[\big(\frac{1-(s-s_1)/m}{\tau}\big)^{1-(s-s_1)/m}+1\big]|v|_s,
\end{split}\]
for every $ v\in h^s$, where $C(s_1,s,0)=0$.
\label{parabolic}
\end{lemma}
\begin{lemma}For any $0\leqslant \tau_1\leqslant \tau_2\leqslant1$, we have 
\[|a(\tau)-a(\tau_1)|_{s_1}\leqslant  C(s_1,s, M_1)\big(|\tau_2-\tau_1|^{\frac{s-s_1}{m}}+|\tau_2-\tau_1|\big).\]
\label{grow}
\end{lemma}
\begin{proof}
For $a(\tau_2)$ we have the  representation below:
\[a(\tau_2)=e^{\mathbb{F}(\tau_2-\tau_1)}a(\tau_1)+\int_{\tau_1}^{\tau_2}e^{\mathbb{F}(\tau_2-\tau)}Y(a(\tau),\tau)ds.\]
The the assertion of the lemma  follows from Lemma \ref{parabolic} and inequity (\ref{rnls-r-lip2}). \end{proof}
Denote $\mathcal{Y}(v,\tau)=Y(v,\tau)-R(v)$. Then by Lemma \ref{rnls-lem-lip}, the relation (\ref{rnls-r-lip2}) also holds for the map $\mathcal{Y}(v,\tau)$.

Now we fix some $$s_1\in(d/2,s).$$ 
The following lemma is the main step of our proof.
\begin{lemma} For $\tilde{\tau}\in[0,1]$, 
\[\Big|\int_{0}^{\tilde{\tau}}\mathcal{Y}(a(\tau),\tau)d\tau\Big|_{s_1}\leqslant C(M_1,s_1,s)[ \epsilon^{(s-s_1)/m}+ \epsilon^{1/2}].
\]
\label{rnls-lem-apr}
\end{lemma}
\begin{proof} Denote by $\mathcal{Y}_{\dk}(v,\tau)$, $\dk\in\mathbb{Z}^d$ the coordinate components of the map $\mathcal{Y}(v,\tau)$.
We first fix some $T_0\in[0,1]$ and divide the time interval $[0,1]$ into subintervals~$[b_l, b_{l-1}]$, $l=1,\cdots, N$, such that :
\[b_0=0, b_l-b_{l-1}=T_0,\quad  \text{for}\quad l=1,\dots, N-1, b_N-b_{N-1}\leqslant T_0,\; a_N=1, \]
where $N\leqslant 1/T_0+1$.
 

For $\tilde{\tau}\in[b_{l}, b_{l+1}]$, let us denote $\mathbf{Y}_{\dk}(\tilde{\tau})=\int_{b_l}^{\tilde{\tau}}\mathcal{Y}_{\dk}(a(\tau),\tau)d\tau$.
Then
\[\mathbf{Y}_{\dk}(\tilde{\tau})=\mathbf{Y}_{\dk}(\tilde{\tau})-\int_{b_{l}}^{\tilde{\tau}}\mathcal{Y}_{\dk}(a(b_{l}), \tau)d\tau+\int_{b_{l}}^{\tilde{\tau}}\mathcal{Y}_{\dk}(a(b_{l}), \tau)d\tau.\]
The last term equals $\mathcal{I}_{\dk}(b_l,\tilde{\tau},q)+i\mathcal{I}_{\dk}(b_l,\tilde{\tau},p)$, with 
\[\begin{split}\mathcal{I}_{\dk}(b_{l},\tilde{\tau},q)=&\sum_{(\dk_1,\dots,\dk_{2q+1})\in S(\dk,q)\setminus \mathcal{R}(\dk,q)}a_{\dk_1}(b_{l})\overline{a_{\dk_2}(b_{l})}\dots a_{\dk_{2q-1}}(b_{l})\overline{a_{\dk_{2q}}(b_{l})}a_{\dk_{2q+1}}(b_{l})\\
&\times\frac{\epsilon}{i[-\lambda_{\dk}+ \sum_{j=1}^{2q+1}(-1)^{j-1}\lambda_{\dk_j}]}\exp\{\epsilon^{-1}i[-\lambda_{\dk}+ \sum_{j=1}^{2q+1}(-1)^{j-1}\lambda_{\dk_j}]\tau\}\Big|_{b_l}^{\tilde{\tau}}.\end{split}\]
Let $$\mathcal{I}(b_l,\tilde{\tau})=\big(\mathcal{I}_{\dk}(b_l,\tilde{\tau},q)+i\mathcal{I}_{\dk}(b_l,\tilde{\tau},p),\dk\in\mathbb{Z}^d\big).$$ Since the quantities $|-\lambda_{\dk}+ \sum_{j=1}^{2q+1}(-1)^{j-1}\lambda_{\dk_j}|$, if do not equal to zero, are alway bigger than 1, hence we have 
\[\max_{\tilde{\tau}\in[b_l,b_{l+1}]}|\mathcal{I}(b_l,\tilde{\tau})|_s\leqslant 2\epsilon\max_{v\in B(M_1)}|P(v)|_s\leqslant\epsilon 2C(M_1).\]
Then choosing $T_0=\epsilon^{1/2}$, using Lemma \ref{grow},  we obtain
\[ \begin{split}\Big|\int_{0}^{\tilde{\tau}}\mathcal{Y}(a(\tau),\tau)d\tau\Big|_{s_1}
&\leqslant \sum_{l=0}^{N-1}\Big\{\int_{b_{l}}^{b_{l+1}}|\mathcal{Y}(a(\tau),\tau)-\mathcal{Y}(a(b_{l}),\tau)|_{s_1}ds+|\mathcal{I}(b_{l},b_{l+1})|_{s_1}\Big\}\\
&\leqslant\sum_{l=0}^{m-1}\Big[\int_{b_l}^{b_{l+1}}C(M_1)|a(\tau)-a(b_l)|_{s_1}d\tau+\epsilon2 C(M_1)]\\
&\leqslant[ (T_0^2+T_0^{1+(s-s_1)/m})2C(M_1)+\epsilon2 C(M_1)](\frac{1}{T_0}+1)\\
&\leqslant C^{\prime}(M_1)[\epsilon^{(s-s_1)/m}+\epsilon^{1/2}].
\end{split}
\]
This proof the assertion of the Lemma.
\end{proof}

Let $a^{\prime}(\tau)$ be a solution of the effective equation (\ref{rnls-effective2}) with initial datum \mbox{$a^{\prime}(0)=a(0)$}. Denote 
$$T^{\prime}=\min\{\tau: \; |a^{\prime}(\tau)|_p\geqslant M_1+1\}\quad\text{and}\quad T'_1=\min\{1,T'\}.$$
For $\tilde{\tau}\in[0, T'_1]$,
\[a(\tilde{\tau})-a'(\tilde{\tau})=\int_0^{\tilde{\tau}}e^{\mathbb{F}(\tilde{\tau}-\tau)}[R(a(\tau))-R(a'(\tau))]d\tau+\int_0^{\tilde{\tau}}e^{\mathbb{F}(\tilde{\tau}-\tau)}\mathcal{Y}(a(\tau),\tau)d\tau.\]
Now we estimate the $|\cdot|_{s_1}$-norm of the last quantity using Lemmas \ref{parabolic} and  \ref{rnls-lem-apr}.
\[\begin{split}&\Big|\int_0^{\tilde{\tau}}e^{\mathbb{F}(\tilde{\tau}-\tau)}\mathcal{Y}(a(\tau),\tau)d\tau\Big|_{s_1}\\
&=\Big|\int_0^{\tilde{\tau}}\mathcal{Y}(a(\tau),\tau)d\tau+\int_0^{\tilde{\tau}}\mathbb{F}\big(e^{\mathbb{F}(\tilde{\tau}-\tau)}\int_0^{\tau}\mathcal{Y}(a(\tau'),\tau')d\tau'\big)d\tau\Big|_{s_1}\\
&\leqslant C(M_1,s_1,s)[\epsilon^{(s-s_1)/m}+\epsilon^{1/2}]+C'(M_1,s_1,s)[\epsilon^{\frac{s-s_1}{2m}}+\epsilon^{1/2}]\int_0^{\tilde{\tau}}\Big(\frac{1}{2\tau}\Big)^{(1-\frac{s-s_1}{2m})}d\tau \\
&\leqslant C''(M_1,s,s_1)[\epsilon^{\frac{s-s_1}{2m}}+\epsilon^{1/2}].
\end{split}\]
So
\[|a^{\prime}(\tilde{\tau})-a(\tilde{\tau})|_{s_1}\leqslant \int_0^{\tilde{\tau}}C(M_1+1,s_1)|a^{\prime}(\tau)-a(\tau)|_{s_1}d\tau+C(M_1,s,s_1)[\epsilon^{\frac{s-s_1}{2m}}+\epsilon^{1/2}].\]
By Gronwall's lemma, we have that 
$$|a^{\prime}(\tilde{\tau})-a(\tilde{\tau})|_{s_1}\leqslant C[\epsilon^{\frac{s-s_1}{2m}}+\epsilon^{1/2}],\quad \tilde{\tau}\in[0, T^{\prime}_1].$$
Let fix any positive $\rho<<1$. For any $s_1<s$,  if  $\epsilon$ is  small enough, then  using the bootstrap argument  we get that  
\[|a'(T'_1)|_{s_1}\leqslant M_1+\rho.\]
Therefore $|a'(T'_1)|_s\leqslant M_1+\rho$, which means that $T'>1$.

Since $I(a(\tau))=I(v(\tau))$, we have 
\[\begin{split}&|I(v(\tau))-I(a^{\prime}(\tau))|_{s_1}^{\sim}=|I(a(\tau))-I(a'(\tau))|^{\sim}_{s_1}\\
&\leqslant |a(\tau)-a'(\tau)|_{s_1}\leqslant C[\epsilon^{\frac{s-s_1}{2m}}+\epsilon^{1/2}],\quad\tau\in[0,1].
\end{split}\]
 This finishes the proof of Theorem \ref{rnls-main}.

\section{ Discussion of Proposition \ref{rnls-p1}}
\subsection{The case $\mu>0$, $b\leqslant0$ and $c\leqslant0$} In this subsection we denote $|\cdot|_s$ to be the $L_s$-norm. We first fix arbitrary  $T>0$ and $m=1$. Let $u(t,x)$ be a solution of Eq. (\ref{rnls1}) with $u(0,x)=u_0$.  Take the $L_2$-scalar product of Eq. (\ref{rnls1}) and $u$,
\[\begin{split}\frac{d}{dt}||u||_0^2&=2\langle u, u_t\rangle=2\big\langle u,-i\Delta u+\epsilon(\mu \Delta u-b|u|^{2p}u+ci|u|^{2q}u)\big\rangle\\
&=\epsilon\big(-2\mu||u||_1^2+2\mu||u||_0^2+2b|u|_{2p+2}^{2p+2}\big)\leqslant 0.
\end{split}\]
Therefore we have 
\begin{equation}\label{l2-norm}
||u(t)||_0\leqslant ||u_0||_0, \quad\text{for}\quad t\in[0,+\infty).\end{equation}
We rewrite Eq. (\ref{rnls1}) as 
\begin{equation}\label{h1-ham}u_t+i(\Delta u-c|u|^{2q}u)=\Delta u+b|u|^{2p}u.\end{equation}
The l.h.s is a hamiltonian system with Hamiltonian function
\[\mathcal{H}(u)=\frac{1}{2}\langle-\Delta u, u\rangle-\frac{c\epsilon}{2q+2}|u|_{2q+2}^{2q+2}.\]
We have $d\mathcal{H}(v)=\langle-\Delta u,v\rangle-c\epsilon\langle |u|^{2q}u,v\rangle$, and if $v$ is the vector field in the l.h.s of Eq. (\ref{h1-ham}), then $d\mathcal{H}(v)=0$. So 
\[\frac{d}{dt}\mathcal{H}(u(t))=\epsilon\langle-\Delta u,\Delta u+b |u|^{2p}u\rangle-\epsilon^2 bc|u|_{2p+2q+2}^{2p+2q+2}-\epsilon^2 c\langle |u|^{2q}u,\Delta u\rangle.\]
Denoting $U_q=\frac{1}{q+1}u^{q+1}$ and $U_p=\frac{1}{p+1}u^{p+1}$, we have 
\[\langle |u|^{2q}u,\Delta u\rangle\leqslant -\int_{\mathbb{T}^d}|\nabla u|^2|u|^{2q}dx=-||\nabla U_q||_0^2.\]
and a similar relation holds for $q$ replaced by $p$. Therefore
\[\frac{d}{dt}\mathcal{H}(u(t))\leqslant0,\]
and \[\mathcal{H}(u(t))\leqslant \mathcal{H}(u(0)),\quad \text{for} \quad t\in[0,\infty).\]
Therefore, 
\[||u(t)||_1\leqslant (\mathcal{H}(u_0)+|u_0|_2^2)^{1/2},\quad t\in[0,\infty).\]
Now we consider 
\begin{equation}\label{h2-norm}\frac{1}{2}\frac{d}{dt}\langle (-\Delta)^2u,u\rangle=\epsilon\langle (-\Delta)^2u, \mu\Delta u+b|u|^{2p}+ic|u|^{2q}u\rangle.
\end{equation}
Using the integration by part and  the H\"older inequality we obtain that
\begin{equation}
\langle (-\triangle)^2u, |u|^{2p}u\rangle\leqslant ||u||_3|(|u|^{2p}\nabla u)|_2\leqslant ||u||_3 |u|_{2pq_1}^{2p}|\nabla u|_{p_1},\label{5.13}
\end{equation}
where $p_1,q_1<\infty$ satisfy $1/p_1+1/q_1=1/2$.   Let $p_1$ and $q_1$ have the form
\[p_1=\frac{2d}{d-2s'},\quad q_1=\frac{d}{s'}.\]
We specify parameter $s$:
for $d\geqslant 3$,   choose $s'=p(d-2)<\min\{d/2,2\}$;
for $d=1,2$, choose $s\in (0,\frac{1}{2})$.  Additionally, we assume 
\begin{equation}
0\leqslant p,q<\infty\quad \text{if}\quad d=1,2\quad\text{and}\quad 0\leqslant p,q<\min\{\frac{d}{2},\frac{2}{d-2}\}\quad \text{if}\quad d\geqslant 3.
\label{cgl1c-2}
\end{equation}
 Then  
the Sobolev embeddings 
\begin{equation}H^{s'}(\mathbb{T}^d)\to L^{p_1}(\mathbb{T}^d) \quad\text{and} \quad H^1(\mathbb{T}^d)\to L^{2pq_1}(\mathbb{T}^d),
\label{sobolev}
\end{equation}
 imply  that 
\[|\nabla u|_{p_1}\leqslant ||u||_{1+s'},\quad |u|_{2pq_1}^{2p}\leqslant ||u||_1^{2p}.\]
Applying  the interpolation and the Young inequality we find that for any $\delta>0$,

\begin{equation}
\begin{split}
-\langle \triangle^2 u , |u|^{2p}u \rangle&\leqslant ||u||_{3}||u||_{1+s'}||u||_1^{2p}\\
&\leqslant C||u||_{3}^{1+\frac{1+s'}{3}}||u||_0^{\frac{2-s'}{3}}||u||_1^{2p}\\
&\leqslant \delta ||u||_{3}^2+C(\delta)(||u||_0^{\frac{2-s'}{3}}||u||_1^{2p})^{\frac{2-s'}{6}},
\end{split}
\end{equation}
We can deal with other terms in (\ref{h2-norm})  similarly. Choosing  suitable   $\delta$, from the inequality above together with (\ref{h2-norm}) we have
\[\frac{1}{2}\frac{d}{dt}||u||_2^2\leqslant -\frac{\mu\epsilon}{2}||u||_3^2+C(2,||u_0||_1).\]
So 
\[||u(t)||_2\leqslant C'(2,||u_0||_2, T), \quad t\in[0,\epsilon^{-1}T].\]
By similar argument, for any $l\geqslant3$, we can obtain
\[||u(t)||_l\leqslant C'(l,||u_0||_l, T), \quad t\in[0,\epsilon^{-1}T]. \]
 The additional condition (\ref{cgl1c-2}) is needed to insure  the validity of the  Sobolev imbeddings~(\ref{sobolev}), which are the key of the proof. It is not difficult to see that for any fixed $p,q,d\in\mathbb{N}$, we always can find $m\in\mathbb{N}$ large enough and suitable Sobolev imbeddings to make the arguments above work.  This confirms the assertion  of Proposition~\ref{rnls-p1}.

\subsection{The case $\mu=b=0$.} For simplicity we assume $c=\pm1$.
Briefly speaking, in this case,  Proposition \ref{rnls-p1} directly follows from the global existence theory for solutions of the nonlinear Schr\"odinger equation
\begin{equation}
\label{rnls3}
u_t+i\Delta u=\pm i\epsilon|u|^{2q}u.
\end{equation}
 The  equation (\ref{rnls3}) has two conservative quantities:
\begin{equation}||u(t)||_0=||u(0)||_0,
\end{equation}
and \[E_q(u(t))=\int_{\mathbb{T}^d}\frac{1}{2}|\nabla u(x,t)|^2dx\pm\frac{\epsilon}{2q+2}\int_{\mathbb{T}^d}|u(x,t)|^{2q+2}dx=E(u(0)).\]
We claim the $H^1$-norm $||u(t)||_1$ remains bounded if the parameter $\epsilon$ is small enough. Indeed, the defocusing case is clear. In the focusing case, we have 
\[\int_{\mathbb{T}^d}|\nabla u(x,t)|^2dx=\frac{\epsilon}{q+1}\int_{\mathbb{T}^d}|u(x,t)|^{2q+2}dx+2E(u(0)).\]
Using the $L^2$-conservation law and the Sobolev embedding:  $$H^1(\mathbb{T}^d)\to L^r, \quad r<\infty \quad \text{and}\quad r\leqslant \frac{2d}{d-2},$$ we obtain for $d$ and $q$ satisfying condition (\ref{rnls-p1-r0}),
\[||u(t)||_1^2\leqslant ||u(0)||_1^2+\epsilon C(q,d)||u(t)||_1^{2q+2}.\]
So $$||u(t)||_1^2\leqslant \frac{||u(0)||_1^2}{1-\epsilon C(q,d)||u(t)||_1^{2q}}.$$
If $\epsilon\leqslant C(q,d)^{-1}2^{4q-1}||u(0)||_1^{-2q}$, we have
\begin{equation}
||u(t)||_1\leqslant C(||u(0)||_1).\label{rnls-p1-bh1}
\end{equation}
Now we give a direct proof of the case $d=2$ and $q=1$, following \cite{brezis1980}. Similar proof also works for  cases $d=1$ and $q\in\mathbb{N}$.
\begin{lemma}  For every $u\in H^2(\mathbb{T}^2)$ with $||u||_1\leqslant1$, we have 
\[||u||_{L^{\infty}}\leqslant C(1+\sqrt{\log(1+||u||_2)}).\]
\label{rnls-p1-lem1}
\end{lemma}
For a proof of this lemma, see  Lemma 2 in \cite{brezis1980}.

\begin{lemma}(Moser's inequality, see e.g. Proposition 3.7 in \cite{pde3})
For $u\in H^2(\mathbb{T}^2)$, we have 
\[|||u|^2u||_2\leqslant C||u||^2_{L^{\infty}}||u||_2.\]
\label{rnls-p1-lem2}
\end{lemma}
\begin{proof}
For $u\in H^2(\mathbb{T}^2)$ we have 
\[|\Delta(|u|^2u)|\leqslant C(|u|^2|\Delta u|+|u||\nabla u|^2),\]
and so
\begin{equation}
|||u|^2u||_2\leqslant C||u||^2_{L^{\infty}}||u||_2+C||u||_{L^{\infty}}(\int_{\mathbb{T}^2}|\nabla u|^4 dx)^{1/2}.
\label{rnls-p1-r1}
\end{equation}
Using the Gagliardo-Nirenberg inequality (see \cite{nirenberg1959}), we have 
\begin{equation}(\int_{\mathbb{T}^2}|\nabla u|^4dx)^{1/2}\leqslant C||u||_{L^{\infty}}||u||_2.
\label{rnls-p1-r2}
\end{equation}
Combining (\ref{rnls-p1-r1}) and (\ref{rnls-p1-r2}) we obtain the statement of the lemma.
\end{proof}
Let us denote by $S(t)$ the $L^2$ isometry $S(t)=e^{-it\Delta}$. Then we have 
\[u(t)=S(t)u(0)+i\epsilon\int_0^tS(t-s)|u(s)|^2u(s)ds.\]
Using Lemmas \ref{rnls-p1-lem1} and \ref{rnls-p1-lem2} and the boundness of $H^1$-norm we have 
\[||u(t)||_2\leqslant ||u(0)||_2+C\epsilon\int_0^t||u(s)||_2[1+\log(1+||u(s)||_2)]ds.\]
So
\[||u(t)||_2\leqslant ||u(0)||_2e^{C_1e^{C_2\epsilon t}}.\]
This verifies the statement of Proposition \ref{rnls-p1} in  this case.
\begin{remark}
The same proof also applies to nonlinear Schr\"odinger equations on~$\mathbb{T}^2$ with other  cubic nonlinearities, e.g. the nonlinearities with Hamiltonians of the forms $\mathcal{H}_3=\int |u|^2(u+\bar{u})dx$ and $\mathcal{H}_3^{\prime}=\int u^3+\bar{u}^3dx$.
\end{remark}
 
 For the other cases, more sophisticated theory is needed. We refer the readers to  the theories of the Cauchy problem for the nonlinear Schr\"odinger equations in~\cite{bourgain1993,herr2011}. From these works we know that  for a solution $u(t)$ of Eq. (\ref{rnls3}) with $u(0)\in H^2$,  there exist   $T_1>0$ and $C_1>0$ that depend only on the bound of the $H^1$-norm $||u(t)||_1$ (inequality (\ref{rnls-p1-bh1})) such that for every  $t_0\in[0,\infty)$, we have
 \[||u(t)||_2^2\leqslant ||u(t_0)||_2^2+\epsilon T_1C_1||u(t_0)||_2^2,\quad t\in [t_0,t_0+T_1].\]
 Therefore \[ ||u(t)||_2\leqslant C(||u(0)||_2)e^{\epsilon tC^{\prime}(||u(0)||_2)}.\]
This confirms the assertion of Proposition \ref{rnls-p1}.

\section*{Acknowledgments}The author wants to thank his Ph.D supervisor professor Sergei Kuksin for formulation of the problem and guidance. He also wants to thank the staff and faculty at C.M.L.S of \'Ecole Polytechnique for their support.

\bibliography{rnls2}

\end{document}